\documentclass[12pt]{article}
\usepackage{amsmath}
\usepackage{amsthm}
\usepackage{enumerate}
\usepackage{graphicx}
\usepackage[hidelinks]{hyperref}
\usepackage{xcolor}
\newtheorem{theorem}{Theorem}
\newtheorem{lemma}{Lemma}
\theoremstyle{definition}
\newtheorem{definition}{Definition}
\newtheorem{conjecture}{Conjecture}
\title{Kernels and Small Quasi-Kernels in Digraphs}
\author{Allan van Hulst \\
\small Unaffiliated \\[-0.8ex]
\small The Netherlands \\
\small\tt allanvanhulst@protonmail.com}
\date{\small \copyright\,\,The author. Released under the CC BY license (International 4.0).}
\begin{document}
\maketitle
\begin{abstract}
  A directed graph $D=(V(D),A(D))$ has a kernel if there exists an
  independent set $K\subseteq V(D)$ such that every vertex $v\in V(D)-K$
  has an ingoing arc $u\mathbin{\longrightarrow}v$ for some $u\in K$. There are
  directed graphs that do not have a kernel (e.g. a 3-cycle). A quasi-kernel
  is an independent set $Q$ such that every vertex can be reached in at most
  two steps from $Q$. Every directed graph has a quasi-kernel. A conjecture
  by P.L. Erd\H{o}s and L.A. Sz\'ekely (cf. A. Kostochka, R. Luo, and, S. Shan,
  \href{https://arxiv.org/abs/2001.04003v1}{\color{blue}arxiv:2001.04003v1}, 2020) postulates 
  that every source-free directed graph has a 
  quasi-kernel of size at most $|V(D)|/2$, where source-free refers to every vertex 
  having in-degree at least one. In this note it is shown that every source-free directed graph 
  that has a kernel also has a quasi-kernel of size at most $|V(D)|/2$, by 
  means of an induction proof. In addition, all definitions and proofs in 
  this note are formally verified by means of the Coq proof assistant.
\end{abstract}

In this note, all directed graphs $D=(V(D),A(D))$ are assumed to be finite
and without self-loops. The notation $u\mathbin{\longrightarrow}v$ is used
to denote $(u,v)\in A(D)$. A set $I\subseteq V(D)$ is \emph{independent} if 
there are no two vertices $u,v\in I$ connected by an arc in any direction. A 
\emph{kernel} $K\subseteq V(D)$ is an independent set such that every vertex 
in $V(D)$ can be reached from a vertex in $K$ in at most one step. A 
\emph{quasi-kernel} $Q\subseteq V(D)$ is a weakening of the concept of kernel by 
requiring that every vertex in $V(D)$ can be reached in at most two steps from
a vertex in $Q$. A source in a directed graph is a vertex $u\in V(D)$ having 
only outgoing arcs. A directed graph is therefore said to be \emph{source-free} 
if every vertex has at least one ingoing arc. 

Based on these definitions, it is clear that every kernel is also a quasi-kernel. 
However, there do exist source-free directed graphs that do not have a kernel,
for instance an odd directed cycle. Moreover, there are source-free directed graphs 
that do have a kernel, but none less than or equal to $|V(D)|/2$ in size, as shown 
by the examples in Figure \ref{fig:examples}. It is straightforward to prove 
that every directed graph has a quasi-kernel \cite{chvatal}.

The following conjecture is attributed to P.L. Erd\H{o}s and L.A. Sz\'ekely
(cf. \cite{fete,kostochka,web}). 

\begin{conjecture}
\label{con:es}
Every source-free digraph has a quasi-kernel of size at most $|V(D)|/2$.
\end{conjecture}

\begin{figure}
\begin{center}
\includegraphics{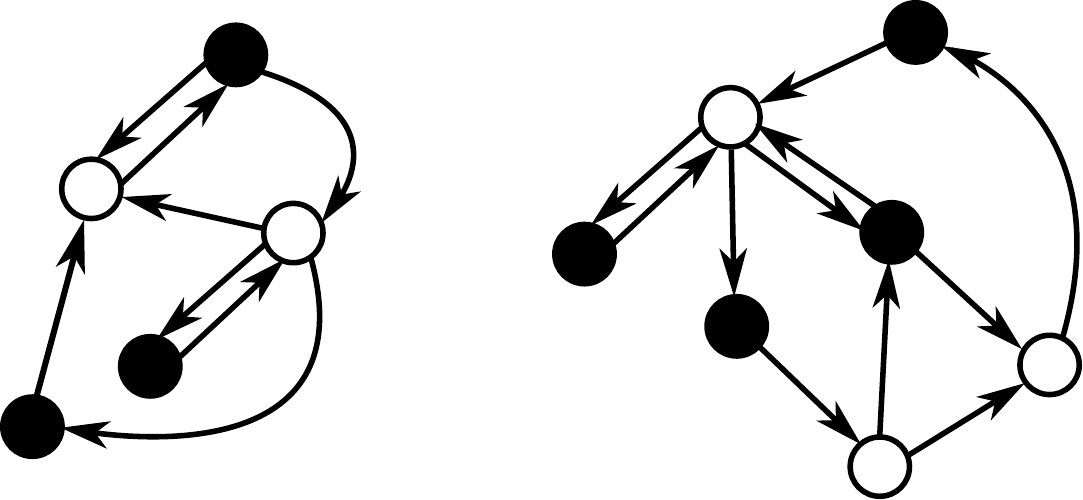}
\end{center}
\caption{Two directed source-free graphs that do not possess a kernel of size 
less than or equal to $|V(D)|/2$. The black vertices indicate smallest possible 
kernels in both examples.}
\label{fig:examples}
\end{figure}

The purpose of this note is to prove Conjecture \ref{con:es} in the presence
of a kernel, by means of an induction proof of moderate technicality. A formal
verification of this proof by means of the Coq proof assistant is also supplied 
with this paper\footnote{This proof has been checked using version 8.4pl4 of the
Coq proof assistant, and is not guaranteed to work in earlier or later versions
of Coq.}.

A kernel in a directed graph is somewhat related to the concept of an 
independent dominating set in the context of undirected graphs \cite{goddard}, 
and some general terminology in this note is inherited from the theory of domination
in undirected graphs \cite{cockayne}.

A digraph has a kernel if and only if it does not have a directed odd
cycle \cite{richardson}. Existence of a quasi-kernel of size at most
$|V(D)|/2$ is guaranteed in case $D$ is a tournament, semicomplete multipartite,
or locally semicomplete \cite{heard}, or an orientation of a graph with
chromatic number at most four \cite{kostochka}. 

The two definitions below are followed by a brief outline of
the solution proposed in this note. It is assumed that $D$ is a directed
graph in these two definitions.

\begin{definition}
\label{def:epon}
For $S\subseteq V(D)$ and $u\in S$, an out-neighbor $v\in V(D)-S$ of $u$ is
said to be an \textit{external private out-neighbor} (\textsc{epon}) with 
regard to $S$, if for all $w\in S$ such that $w\mathbin{\longrightarrow}v$ it holds 
that $u=w$.
\end{definition}

As an example, the right-most vertex in Figure \ref{fig:examples} is an
\textsc{epon} with regard to the black vertices. The directed graph on the 
left does not have an \textsc{epon}.

\begin{definition}
\label{def:in_dom}
A quasi-kernel $Q\subseteq V(D)$ is \textit{inward dominated} if for all
$w\in Q$ and $v\in V(D)-Q$ such that $v\mathbin{\longrightarrow}w$, there exists a
$u\in Q$ such that $u\mathbin{\longrightarrow}v$.
\end{definition}

The induction proof in Theorem \ref{thm:main} is outlined briefly here.
First, Lemma \ref{lem:in_dom} is used to derive existence of an inward
dominated quasi-kernel $K$, based on the assumption of a kernel being
present. If every vertex $u\in K$ has an \textsc{epon}, then the conclusion 
$|K|\leq|V(D)|/2$ follows directly from Lemma \ref{lem:epon_leq}. If there 
is some vertex $u\in K$ that does not have an \textsc{epon}, then the vertex 
can be removed, and this process clearly terminates in finitely many steps. 
The technicality of this solution lies entirely in proving that the premisses 
for the induction hypothesis are satisfied. In particular, it is not entirely 
trivial to show that $K-\{u\}$ is again an inward dominated quasi-kernel.

Theorem \ref{thm:main} is supported by the three lemmas below. These are
proved here in detailed form, in order to achieve closer resemblance to
the way they are formalized in the Coq proof. It is assumed that $D$ is a 
directed graph in these three lemmas.

\begin{lemma}
\label{lem:in_dom}
If $K\subseteq V(D)$ is a kernel then $K$ is an inward dominated quasi-kernel.
\end{lemma}
\begin{proof}
From the definition it is immediate that $K$ is a quasi-kernel. Assume
that $w\in K$ and $v\in V(D)-K$ such that $v\mathbin{\longrightarrow}w$. Then,
since $K$ is a kernel, there exists a $u\in K$ such that 
$u\mathbin{\longrightarrow}v$. Lemma \ref{lem:in_dom} is formalized as
\texttt{Lemma kernel\_qkernel} and \texttt{Lemma kernel\_in\_dom} in
the Coq proof.
\end{proof}

\begin{lemma}
\label{lem:epon_incl}
If $u$ has an \textsc{epon} with regard to $T\subseteq V(D)$ and
$S\subseteq T$ such that $u\in S$ then $u$ has an \textsc{epon} with 
regard to $S$.
\end{lemma}
\begin{proof}
Suppose that $u\in S$ has an \textsc{epon} $v\in V(D)-T$ with regard to 
$T$ and assume towards a contradiction that there exists a $w\in S$ such 
that $w\not=u$ and $w\mathbin{\longrightarrow}v$. Then clearly $w\in T$ and therefore
$v$ cannot be an \textsc{epon} with regard to $T$. Lemma \ref{lem:epon_incl}
is encoded as \texttt{Lemma has\_epon\_incl} in the Coq proof.
\end{proof}

\begin{lemma}
\label{lem:epon_leq}
If all vertices in $S\subseteq V(D)$ have an \textsc{epon} with regard 
to $S$, then $|S|\leq |V(D)|/2$.
\end{lemma}
\begin{proof}
In general, say that $R\subseteq X\times Y$ is a binary total injective relation if
\begin{enumerate}[(1)]
\item for all $x\in X$ there exists a $y\in Y$ such that $R\,(x,y)$, and
\item for all $x,x'\in X$ and $y\in Y$ such that $R\,(x,y)$ and $R\,(x',y)$ it
      holds that $x=x'$.
\end{enumerate}
Clearly, $|X|\leq |Y|$ holds here due to injectivity. 

Say that $T\subseteq V(D)$ contains the vertices that are an \textsc{epon} 
with regard to the set $S$, as given in the statement of this lemma. Define 
$R\subseteq
S\times T$ as
\begin{center}
\begin{math}
R=\{(u,v)\in S\times T\mid 
  v\,\,\textrm{is an \textsc{epon} of}\,\,u\,\,\textrm{with regard to}\,\,S\}
\end{math}
\end{center}
\noindent and observe that $R$ is a binary total injective relation. It then 
follows that $|S|\leq |T|$ and hence, as $S$ and $T$ are disjoint, 
$|S|+|T|\leq |V(D)|$ and thus $2|S|\leq |V(D)|$. Lemma \ref{lem:epon_leq}
is encoded as \texttt{Lemma all\_epon\_half\_size} (using \texttt{Lemma inj\_leq})
in the Coq proof.
\end{proof}

Conjecture \ref{con:es}, in case a kernel is present, is proved in 
Theorem \ref{thm:main}. This part of the proof is formalized as 
\texttt{Theorem main} at the end of the Coq code.

\begin{theorem}
\label{thm:main}
If a source-free digraph $D$ has a kernel, then $D$ has a quasi-kernel 
of size at most $|V(D)|/2$.
\end{theorem}
\begin{proof}
Assume $D$ is a source-free directed graph that has a kernel $K$. From
Lemma \ref{lem:in_dom} it is clear that $K$ is also an inward dominated
quasi-kernel. By induction, the following will be shown: if $K$ is an
inward dominated quasi-kernel, then there exists a quasi-kernel of size
at most $|V(D)|/2$. For this purpose, define $S\subseteq V(D)$ as follows:
\begin{center}
\begin{math}
S = \{u\in K\mid u\,\,\textrm{does not have an \textsc{epon} with regard to}\,\,K\},
\end{math}
\end{center}
\noindent and apply induction towards $|S|$, thereby generalizing over all 
other variables. 

If $|S|=0$ then every vertex in $K$ has an \textsc{epon} and by Lemma 
\ref{lem:epon_leq} it then follows that $|K|\leq |V(D)|/2$.

Assume that $|S|>0$ and assume that there exists a vertex $u\in S$ 
such that $u$ does not have an \textsc{epon} with regard to $S$ in $D$.
Now define $R\subseteq K-\{u\}$ as follows:
\begin{center}
\begin{math}
R = \{v\in K-\{u\}\mid v\,\,\textrm{does not have an \textsc{epon} with regard to}\,\,K-\{u\}\}.
\end{math}
\end{center}

Now, three premisses are required to be able to apply the induction
hypothesis and thereby complete the proof: (1) $|R|<|S|$, (2) $K-\{u\}$ 
is a quasi-kernel, and, (3) $K-\{u\}$ is inward dominated. These will
be proved here one by one.
\begin{enumerate}[(1)]
\item Clearly, it holds that $u\in S$ and $u\not\in R$, and therefore
      it suffices to show that $R\subseteq S$. Assume that $v\in R$ such
      that $v$ does not have an \textsc{epon} with regard to $K-\{u\}$.
      Then, by contraposition of Lemma \ref{lem:epon_incl}, $v$ cannot
      have an \textsc{epon} with regard to $K$, hence $v\in S$.
\item It is immediate that $K-\{u\}$ is also independent. Assume that
      $v\in V(D)$ and distinguish between the following cases.
      \begin{itemize}
      \item If $u=v$ then, as $D$ is source-free, there must exist a vertex  
            $u'\in V(D)$ such that $u'\mathbin{\longrightarrow}u$ and $u'\not\in K$. As $K$ is an
            inward dominated quasi-kernel, there must exist a vertex $w\in K$
            such that $w\mathbin{\longrightarrow}u'$. Clearly, only the case $w=u$
            is relevant. Assume there does not exist some alternative vertex 
            $w'\in K$ such that $w'\mathbin{\longrightarrow}u'$. Then, $u'$ is an
            \textsc{epon} of $u$ with regard to $K$, thereby contradicting
            the assumption $u\in S$.
      \item Now consider the case $u\not=v$, and distinguish between the
            cases following from the assumption that $K$ is a quasi-kernel.
            \begin{itemize}
            \item If $v\in K$ then $v\in K-\{u\}$. 
            \item Now suppose that $v\not\in K$ and there exists a vertex $w\in K$ 
                  such that $w\mathbin{\longrightarrow}v$. If $w=u$ then either there
                  exists an alternative vertex $w'\in K$ such that $w'\mathbin{\longrightarrow}v$,
                  or $v$ is an \textsc{epon} of $u$ with regard to $K$.
            \item For the final case corresponding to $v\not\in K$, assume there exist 
                  vertices $x\in K$ and $w\not\in K$ such that $x\mathbin{\longrightarrow}w$ and 
                  $w\mathbin{\longrightarrow}v$. If $x=u$ then either there exists an alternative 
                  vertex $x'\in K$ such that $x'\mathbin{\longrightarrow}w$, or $w$ is an 
                  \textsc{epon} of $u$ with regard to $K$.
            \end{itemize}
      \end{itemize}
\item Assume that $w\in V(D)-(K-\{u\})$ and $v\in K-\{u\}$ such that
      $w\mathbin{\longrightarrow}v$. Since $K$ is inward-dominated, there must
      exist some vertex $x\in K$ such that $x\mathbin{\longrightarrow}w$. Now
      assume $x=u$. If there does not exist some alternative vertex $y\in K$
      such that $y\mathbin{\longrightarrow}w$, then $w$ is an \textsc{epon} of $u$
      with regard to $K$, contradicting $u\in S$.
\end{enumerate}
\end{proof}

Conjecture \ref{con:es} remains open if there is no kernel present. Results
concerning the presence of multiple distinct quasi-kernels in this situation
are known \cite{jacob, gutin}. However, these distinct quasi-kernels are
not necessarily disjoint. The reasonable complexity of the proof in this
note leads to the suggestion that Conjecture \ref{con:es}, in its 
unconditional form, is perhaps also accessible via elementary methods.

\end{document}